\documentclass[english]{article}
\usepackage[T1]{fontenc}
\usepackage[latin9]{inputenc}
\usepackage{float}
\usepackage{amsmath}
\usepackage{amsthm}
\usepackage{amssymb}

\makeatletter
\theoremstyle{plain}
\newtheorem{thm}{\protect\theoremname}
  \theoremstyle{definition}
  \newtheorem{defn}[thm]{\protect\definitionname}
  \theoremstyle{remark}
  \newtheorem{rem}[thm]{\protect\remarkname}
  \theoremstyle{definition}
  \newtheorem{example}[thm]{\protect\examplename}
  \theoremstyle{plain}
  \newtheorem{lem}[thm]{\protect\lemmaname}
  \theoremstyle{plain}
  \newtheorem{cor}[thm]{\protect\corollaryname}

\usepackage{tikz}
\usetikzlibrary{decorations.pathreplacing}

\makeatother

\usepackage{babel}
  \providecommand{\corollaryname}{Corollary}
  \providecommand{\definitionname}{Definition}
  \providecommand{\examplename}{Example}
  \providecommand{\lemmaname}{Lemma}
  \providecommand{\remarkname}{Remark}
\providecommand{\theoremname}{Theorem}

\begin{document}

\title{Tangent Cones to TT Varieties}

\author{Benjamin Kutschan}
\maketitle
\begin{abstract}
As already done for the matrix case in \cite[p.256]{harris}, \cite[Thm. 6.1, p.1872]{absil_graph_similarity}
and \cite[Thm. 3.2]{uschmajew_schneider_matrix_case} we give a parametrization
of the Bouligand tangent cone of the variety of tensors of bounded
TT rank. We discuss how the proof generalizes to any binary hierarchical
format. The parametrization can be rewritten as an orthogonal sum
of TT tensors. Its retraction onto the variety is particularly easy
to compose. We also give an implicit description of the tangent cone
as the solution of a system of polynomial equations.
\end{abstract}

\section{Introduction}

An algebraic variety is defined to be the set of solutions to a system
of polynomial equations. See \cite{coxLittleShea} for a detailed
textbook on algebraic varieties. It is well known that the set of
tensors of bounded TT rank is an algebraic variety. It is generated
by the determinants of minors whose size is the rank of the corresponding
matricizations plus $1$. In smooth points of the variety the tangent
cone is a linear subspace and is also called tangent plane or tangent
space. Even in singular points the tangent cone is an algebraic variety
itself. It can be computed using Gröbner bases as described in \cite[§ 9.7 p. 498 bottom]{coxLittleShea}.
This algorithm yields an implicitely defined tangent cone. Finding
a parametrization (in the context of algebraic geometry, parametrization
means by polynomials) of an algebraic variety in general and of the
tangent cone in particular is a more delicate matter. Even though
there is no general algorithm to determine the parametrization, there
is an algorithm to determine, whether a given parametrization produces
an implicitely defined variety. This process is called Implicitization.
It can also be done using Gröbner bases and is discussed in the textbook
\cite[§ 3.3, p. 128]{coxLittleShea}. Even for varieties with few
defining polynomials and few variables Gröbner bases tend to be very
large. Calculating the tangent cone (in $\mathbb{C}$) and determining
whether our guess is the correct parametrization worked for the variety
of $3\times3\times3$ TT tensors. We used Macaulay2 \cite{M2}. However
other non-trivial examples beyond dimensions $4\times4\times4$ are
intractible with this method as the size of the Gröbner bases produced
appears to grow beyond any reasonable amount of memory. Instead of
using Gröbner bases, it turns out that we can parametrize the tangent
cone of TT varieties by exploiting orthogonality.
\begin{defn}
A \textit{tensor} $A$ is an element of the tensor space $\mathbb{R}^{n_{1}\times...\times n_{d}}$
where $d$ is called the \textit{order} and $n_{i}$ is called the
\textit{dimension} (in the direction) of order $i$.
\end{defn}
\begin{rem}
Note that we can canonically identify the spaces $\mathbb{R}^{n\times m}$
and $\mathbb{R}^{n\cdot m}$ and we will do so throughout this paper
without explicitely stating. We write $A^{(n_{1}...n_{i})\times(n_{i+1}...n_{d})}$
for the matricization (i.e. combining several indices into one using
e.g. lexicographic order) of $A\in\mathbb{R}^{n_{1}\times...\times n_{d}}$
and $A^{n_{1}\times...\times n_{d}}$ for the tensorization. Define
the shorthand $A^{L}:=A^{n_{1}\times(n_{2}...n_{d})}$ and $A^{R}:=A^{(n_{1}...n_{d-1})\times n_{d}}$.
If it is clear from the context we will often omit the matricization
notation.
\end{rem}
Throughout this paper we will use the \textit{TT product} defined
below. In the matrix case it is equivalent to the matrix product and
it allows us with little effort to rigorously describe tensor diagrams.
Even though we do not use tensor diagrams in this work, figure \ref{fig:diagrams}
shall serve as a dictionary to aid those familiar with tensor diagrams.
\begin{figure}[h]
\caption{\label{fig:diagrams}tensor diagrams}
\begin{tikzpicture}[scale=0.2]
\node[anchor=south] at (0,0) {$A_1$}; \draw(4,0)--(0,0)--(0,-2); \draw[black,fill=black] (0,0) circle (1.8ex);
\begin{scope}[shift={(4,0)}] \node[anchor=south] at (0,0) {$A_2$}; \draw(4,0)--(0,0)--(0,-2); \draw[black,fill=black] (0,0) circle (1.8ex); \end{scope}
\begin{scope}[shift={(8,0)}] \node[anchor=south] at (0,0) {$A_3$}; \draw(0,0)--(0,-2); \draw[black,fill=black] (0,0) circle (1.8ex); \end{scope}
\begin{scope}[shift={(14,2)}]
\node[] at (0,-3) {$=A_1A_2A_3$};
\end{scope}
\begin{scope}[shift={(25,1)}] %ein element
\node[anchor=north] at (0,-4) {$A_1$};
\node[anchor=south] at (0,0) {$B_1$};
\draw(4,0)--(0,0)--(0,-4)--(4,-4);
\draw[black,fill=black] (0,0) circle (1.8ex); \draw[black,fill=black] (0,-4) circle (1.8ex);
\begin{scope}[shift={(4,0)}] \node[anchor=north] at (0,-4) {$A_2$}; \node[anchor=south] at (0,0) {$B_2$}; \draw(4,0)--(0,0)--(0,-4)--(4,-4); \draw[black,fill=black] (0,0) circle (1.8ex); \draw[black,fill=black] (0,-4) circle (1.8ex); \end{scope}
\begin{scope}[shift={(8,0)}] \node[anchor=north] at (0,-4) {$A_3$}; \node[anchor=south] at (0,0) {$B_3$}; \draw(2,0)--(0,0)--(0,-4)--(2,-4); \draw[black,fill=black] (0,0) circle (1.8ex); \draw[black,fill=black] (0,-4) circle (1.8ex); \end{scope}
\begin{scope}[shift={(22,1)}] \node[] at (0,-3) {$=((A_1A_2A_3)^R)^T(B_1B_2B_3)^R$}; \end{scope}
\end{scope} \end{tikzpicture}
\end{figure}

\begin{defn}
We define a scalar product on $\mathbb{R}^{n_{1}\times...\times n_{d}}$
as the standard scalar product on $\mathbb{R}^{n_{1}...n_{d}}$. This
induces a norm and the notion of orthogonality. We denote the \textit{TT
product} of the two tensors $A\in\mathbb{R}^{n_{1}\times...\times n_{i}\times k}$
and $B\in\mathbb{R}^{k\times n_{i+1}\times...\times n_{d}}$ by

\[
AB:=\left(A^{R}B^{L}\right)^{n_{1}\times...\times n_{d}}\in\mathbb{R}^{n_{1}\times...\times n_{d}}
\]
Its entries are
\[
(AB)(j_{1},...,j_{d}):=\sum_{m=1}^{k}A(j_{1},...,j_{i},m)B(m,j_{i+1},...,j_{d}).
\]
\end{defn}
Note that the TT product is associative. It is equivalent to the matrix
product if $A$ and $B$ are matrices.
\begin{defn}
Define the variety of TT tensors \cite{oseledets_TT} of order $d$
and dimensions $(n_{1},...,n_{d})$ of rank bounded by $\mathbf{k}=(k_{1},...,k_{d-1})$
as
\[
\mathcal{M}_{\leq(k_{1},...,k_{d-1})}^{n_{1}\times...\times n_{d}}:=\{A\in\mathbb{R}^{n_{1}\times...\times n_{d}}:\forall i:\operatorname{rank}\left(A^{(n_{1}...n_{i})\times(n_{i+1}...n_{d})}\right)\leq k_{i}\}
\]
and the manifold of TT tensors of order $d$ and dimensions $(n_{1},...,n_{d})$
of rank exactly $(k_{1},...,k_{d-1})$ as
\[
\mathcal{M}_{=(k_{1},...,k_{d-1})}^{n_{1}\times...\times n_{d}}:=\{A\in\mathbb{R}^{n_{1}\times...\times n_{d}}:\forall i:\operatorname{rank}\left(A^{(n_{1}...n_{i})\times(n_{i+1}...n_{d})}\right)=k_{i}\}.
\]
\end{defn}
Note that the variety of TT tensors of bounded rank $(k_{1},...,k_{d-1})$
is indeed an algebraic variety. Its defining polynomials are the determinants
of $(k_{i}+1)\times(k_{i}+1)$-minors of $A^{(n_{1}...n_{i})\times(n_{i+1}...n_{d})}$.
A proof for the TT manifold being a manifold can be found in \cite{Uschmajew.Vandereycken:2013}.
\begin{defn}
Define the tangent cone (also known as Bouligand contingent cone or
tangent semicone $C^{+}$ in \cite{oshea_wilson_tangent_semicones})
of an algebraic variety $\mathcal{M}\in\mathbb{R}^{N}$ at a (possibly
singular) point $A\subset\mathcal{M}$ as in \cite{uschmajew_schneider_matrix_case}
and \cite{oshea_wilson_tangent_semicones} as the set of all vectors
that are limits of secants through $A$:
\[
T_{A}\mathcal{M}:=\{\xi\in\mathbb{R}^{N}:\exists(x_{n})\subset\mathcal{M},(a_{n})\subset\mathbb{R}^{+}\text{ s.t. }x_{n}\rightarrow A,\ a_{n}(x_{n}-A)\rightarrow\xi\}.
\]
\end{defn}
\begin{rem}
Even though this will not affect the current work, we want to remark,
that in the complex setting this tangent cone is equivalent to the
algebraic tangent cone. See \cite{coxLittleShea,oshea_wilson_tangent_semicones}.
But we do not know of any proof of the corresponding statement for
the real case.
\end{rem}
A direct consequence from our parametrization will be, that in the
case of TT varieties the $a_{n}$ do not need to be positive. The
following example is included to address a certain peculiarity. In
contrast to Differential Geometry the description of the tangent cone
does not need all smooth arcs, but only analytic arcs. However the
set of first derivatives of analytic arcs $\left\{ v\in\mathbb{R}^{N}:\exists\gamma:[0,\varepsilon]\rightarrow\mathcal{M}\,\text{analytic}:\gamma(0)=A,\dot{\gamma}(0)=v\right\} $
does not suffice. To describe the set of directions of analytic arcs
we need to include the higher order derivatives.
\begin{example}
\label{exa:analytic curves}Consider the variety $\mathcal{M}:=\left\{ (x,y)\in\mathbb{R}^{2}:x^{2}=y^{3}\right\} $
and an analytic arc $\gamma$ with values in $\mathcal{M}$ such that
$\gamma(0)=(0,0)$. Then $\dot{\gamma}(0)$ always vanishes. Verify
this by plugging the analytic arc
\[
\gamma:t\mapsto(a_{1}t+a_{2}t^{2}+...,b_{1}t+b_{2}t^{2}+...)
\]
into the defining equation and compare coefficients. But the tangent
cone of $\mathcal{M}$ at $(0,0)$ is more than $\{(0,0)\}$. This
example also works in the complex case. Note that e.g. $\gamma:t\mapsto(t^{\frac{3}{2}},t)$
is not analytic.
\end{example}
In general the tangent cone can be defined by the first non-zero derivatives
\[
\{v\in\mathbb{R}^{N}:\exists n\in\mathbb{N},\gamma:[0,\varepsilon]\rightarrow\mathcal{M}\,\text{analytic}:
\]
\[
\gamma(0)=A,\,\gamma^{(n)}(0)=v\text{ and }\forall i<n:\gamma^{(i)}(0)=0\}
\]
of analytic arcs through the singular point. See \cite{oshea_wilson_tangent_semicones}
for a detailed discussion. Keeping in mind that any complex variety
can be rewritten as a real variety, this also works in the complex
case. In example \ref{exa:analytic curves} second derivatives suffice.
We will show in Corollary \ref{cor:first_der_suffice} that for the
TT variety first derivatives produce the tangent cone.

Lemmata 8 and 9 are trivial but essential for the proof of our main
result.
\begin{lem}
\[
\mathcal{M}_{\leq(k_{1},k_{2})}^{n_{1}\times n_{2}\times n_{3}}=\mathcal{M}_{\leq k_{1}}^{n_{1}\times(n_{2}n_{3})}\cap\mathcal{M}_{\leq k_{2}}^{(n_{1}n_{2})\times n_{3}}
\]
\end{lem}
\begin{proof}
by definition.
\end{proof}
On a subset we can only define a subset of the secants and thus a
subset of the tangents.
\begin{lem}
\label{lem:T(cup)=00003DCup(T)}For every $A\in\mathcal{M}_{\leq(k_{1},k_{2})}^{n_{1}\times n_{2}\times n_{3}}$
we have
\[
T_{A}\mathcal{M}_{\leq(k_{1},k_{2})}^{n_{1}\times n_{2}\times n_{3}}\subset T_{A}\mathcal{M}_{\leq k_{1}}^{n_{1}\times(n_{2}n_{3})}
\]
 and thus 
\[
T_{A}\mathcal{M}_{\leq(k_{1},k_{2})}^{n_{1}\times n_{2}\times n_{3}}\subset T_{A}\mathcal{M}_{\leq k_{1}}^{n_{1}\times(n_{2}n_{3})}\cap T_{A}\mathcal{M}_{\leq k_{2}}^{(n_{1}n_{2})\times n_{3}}
\]
 
\end{lem}
\begin{proof}
by definition.
\end{proof}
\begin{defn}
Define the range of $A\in\mathbb{R}^{n_{1}\times...\times n_{i}\times k}$
as
\[
\operatorname{range}(A):=\{a\in\mathbb{R}^{n_{1}\times...\times n_{i}}:\exists b\in\mathbb{R}^{k\times1}:a=Ab\}.
\]
\end{defn}

\section{Parametrization of the tangent cone}

We will recall the matrix case as a guiding example and as a necessary
prerequisite. Along the way, we will introduce all proof ideas needed
for the general case. Consider the matrix variety
\[
\mathcal{M}_{\leq k+s}^{n\times m},\quad s>0
\]
i.e. the set of $n\times m$ matrices of rank at most $k+s$. Let
$A\in\mathbb{R}^{n\times k}$ and $B\in\mathbb{R}^{k\times m}$ have
full rank. Then $AB$ has rank $k$ and is a singular point of $\mathcal{M}_{\leq k+s}^{n\times m}$.
As for example shown in \cite{uschmajew_schneider_matrix_case} (compare
also to \cite[p.256]{harris}), any tangent vector in the tangent
cone at $AB$ can be decomposed as
\[
\mathcal{X}=AY+XB+UV=\left(\begin{array}{ccc}
A & U & X\end{array}\right)\left(\begin{array}{c}
Y\\
V\\
B
\end{array}\right)
\]
with $U\in\mathbb{R}^{n\times s}$ and $V\in\mathbb{R}^{s\times m}$.
The converse is true by the following: The analytic arc
\[
\gamma:t\mapsto\left(\begin{array}{cc}
A+tX & tU\end{array}\right)\left(\begin{array}{c}
B+tY\\
V
\end{array}\right)
\]
lies in $\mathcal{M}_{\leq k+s}^{n\times m}$ and its derivative is
$\dot{\gamma}(0)=AY+XB+UV$. Use $\left(\gamma\left(\frac{1}{m}\right)\right)_{\mathbb{N}\ni m\geq N}$
to see, that $\dot{\gamma}(0)$ lies in the tangent cone. We can assume
$A^{T}X=0$ (i.e. the columns of $X$ are orthogonal to the columns
of $A$), $VB^{T}=0$ and either $A^{T}X=0$ or $YB^{T}=0$ by the
following argument. $P_{A}:=AA^{\dagger}$ is the orthogonal projector
onto $\operatorname{range}(A)$, where $A^{\dagger}$ denotes the
Moore-Penrose Pseudoinverse. Defining $\dot{U}:=A^{\dagger}U$ and
$\hat{U}:=(I-P_{A})U$ we can decompose
\begin{equation}
U=P_{A}U+(I-P_{A})U=AA^{\dagger}U+\hat{U}=A\dot{U}+\hat{U}\label{eq:pseudoinverses}
\end{equation}
where $\hat{U}$ is orthogonal to $A$, i.e. $A^{T}\hat{U}=0$. Decomposing
$V$ and $X$ in the same way, we can write $\mathcal{X}=AY+(A\dot{X}+\hat{X})B+(A\dot{U}+\hat{U})(\hat{V}+\dot{V}B)=A(Y+\dot{X}B+\dot{U}\hat{V}+\dot{U}\dot{V}B)+\hat{X}B+\hat{U}\hat{V}$.
We can furthermore assume $U$ and $V$ to have full rank by choosing
them from $\mathbb{R}^{n\times\tilde{s}}$ and $\mathbb{R}^{\tilde{s}\times m}$
respectively with $\tilde{s}$ minimal. We introduce a definition
for this, because we will need it in the tensor case.
\begin{defn}
Let $A\in\mathbb{R}^{n\times m}$ be a matrix of rank $k$ and $A_{1}\in\mathbb{R}^{n\times k}$,
$A_{2}\in\mathbb{R}^{k\times m}$ be such that $A=A_{1}A_{2}$. Call
for the purpose of this paper
\[
A_{1}Y+XA_{2}+UV
\]
 an \textit{$s$-decomposition} of the matrix $\mathcal{X}\in\mathbb{R}^{n\times m}$
(not every matrix is decomposable in this way) if $U\in\mathbb{R}^{n\times s}$,
$V\in\mathbb{R}^{s\times m}$, $A_{1}^{T}X=0$, $A_{1}^{T}U=0$, $VA_{2}^{T}=0$
and $U$ and $V$ have full rank.
\end{defn}
As a first step, we will prove the converse of our main result as
the proof is completely analogous to the matrix case.
\begin{lem}
\label{lem:subset_of_tangent_cone}Assume $A\in\mathcal{M}_{=(k_{1},...,k_{d-1})}^{n_{1}\times...\times n_{d}}$,
i.e. there are $A_{1}\in\mathbb{R}^{n_{1}\times k_{1}}$, $A_{i}\in\mathbb{R}^{k_{i-1}\times n_{i}\times k_{i}}\;\forall i=2,...,d-1$
and $A_{d}\in\mathbb{R}^{k_{d-1}\times n_{d}}$ such that $A=A_{1}...A_{d}$.
If a vector $\mathcal{X}$ can be factorized as
\[
\left(\begin{array}{ccc}
A_{1} & U_{1} & X_{1}\end{array}\right)\left(\begin{array}{ccc}
A_{2} & U_{2} & X_{2}\\
0 & Z_{2} & V_{2}\\
0 & 0 & A_{2}
\end{array}\right)...\left(\begin{array}{ccc}
A_{d-1} & U_{d-1} & X_{d-1}\\
0 & Z_{d-1} & V_{d-1}\\
0 & 0 & A_{d-1}
\end{array}\right)\left(\begin{array}{c}
X_{d}\\
V_{d}\\
A_{d}
\end{array}\right)
\]
with block matrix dimensions $(k_{i}+s_{i}+k_{i})\times(k_{i+1}+s_{i+1}+k_{i+1})$
then it is in the tangent cone of $\mathcal{M}_{\leq(k_{1}+s_{1},...,k_{d-1}+s_{d-1})}^{n_{1}\times...\times n_{d}}$
at $A_{1}...A_{d}$.
\end{lem}
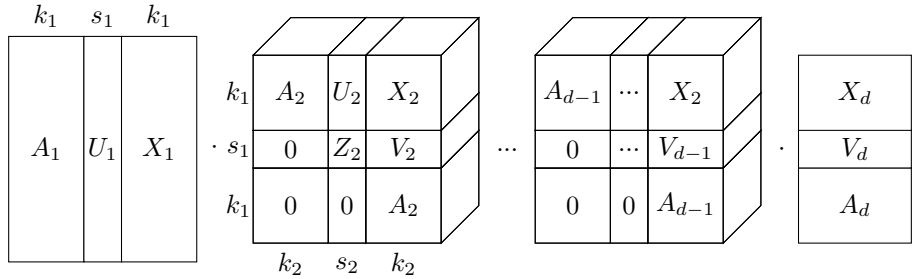
\begin{figure}[h]

\caption{\label{fig:Decomposition}Decomposition of a tangent vector}
\begin{tikzpicture}[scale=0.5]
\newcommand{\cuboid}[6]{ \begin{scope}[shift={(#4,#5)}] \draw[fill=white](#1,#2)--(#1,0)--(0,0)--(0,#2)--cycle; \draw[fill=white](#1,0)--(#1+#3/2,#3/2)--(#1+#3/2,#2+#3/2)--(#1,#2)--cycle; \draw[fill=white](#1+#3/2,#2+#3/2)--(#3/2,#2+#3/2)--(0,#2)--(#1,#2)--cycle; \draw[draw=black](#1,#2)--(#1,0)--(0,0)--(0,#2)--(#1,#2)--(#1+#3/2,#2+#3/2); \draw[draw=black](#1,0)--(#1+#3/2,#3/2)--(#1+#3/2,#2+#3/2)--(#3/2,#2+#3/2)--(0,#2); \end{scope} }
\newcommand{\benSquare}[5]{ \begin{scope}[shift={(#3,#4)}] \draw[fill=white](#1,#2)--(#1,0)--(0,0)--(0,#2)--cycle; \end{scope} }
%FIRST MATRIX
\benSquare{2}{6}{0}{0}; \node at (1,3) {$A_1$}; \node [anchor = south] at (1,6) {$k_1$};
\benSquare{1}{6}{2}{0}; \node at (2.5,3) {$U_1$}; \node [anchor = south] at (2.5,6) {$s_1$};
\benSquare{2}{6}{3}{0}; \node at (4,3) {$X_1$}; \node [anchor = south] at (4,6) {$k_1$};
%FIRST TENSOR
\begin{scope}[shift = {(0.5,0.5)}] \cuboid{2}{2}{2}{6}{0}; \node at (7,1) {$0$}; \node [anchor = east] at (6.2,1) {$k_1$}; \node [anchor = north] at (7,0) {$k_2$};
\cuboid{1}{2}{2}{8}{0}; \node at (8.5,1) {$0$}; \node [anchor = north] at (8.5,-0.15) {$s_2$};
\cuboid{2}{2}{2}{9}{0}; \node at (10,1) {$A_2$}; \node [anchor = north] at (10,0) {$k_2$};
\begin{scope}[shift = {(0,2)}] \cuboid{2}{1}{2}{6}{0}; \node at (7,0.5) {$0$}; \node [anchor = east] at (6.2,0.5) {$\cdot \,\, s_1$};
\cuboid{1}{1}{2}{8}{0}; \node at (8.5,0.5) {$Z_2$};
\cuboid{2}{1}{2}{9}{0}; \node at (10,0.5) {$V_2$}; \end{scope}
\begin{scope}[shift = {(0,3)}] \cuboid{2}{2}{2}{6}{0}; \node at (7,1) {$A_2$}; \node [anchor = east] at (6.2,1) {$k_1$};
\cuboid{1}{2}{2}{8}{0}; \node at (8.5,1) {$U_2$};
\cuboid{2}{2}{2}{9}{0}; \node at (10,1) {$X_2$}; \end{scope}
\end{scope}
%LAST TENSOR
\begin{scope}[shift = {(8,0.5)}] \cuboid{2}{2}{2}{6}{0}; \node at (7,1) {$0$};
\cuboid{1}{2}{2}{8}{0}; \node at (8.5,1) {$0$};
\cuboid{2}{2}{2}{9}{0}; \node at (10,1) {$A_{d-1}$};
\begin{scope}[shift = {(0,2)}] \cuboid{2}{1}{2}{6}{0}; \node at (7,0.5) {$0$}; \node [anchor = east] at (5.8,0.5) {$...$};
\cuboid{1}{1}{2}{8}{0}; \node at (8.5,0.5) {$...$};
\cuboid{2}{1}{2}{9}{0}; \node at (10,0.5) {$V_{d-1}$}; \end{scope}
\begin{scope}[shift = {(0,3)}] \cuboid{2}{2}{2}{6}{0}; \node at (7,1) {$A_{d-1}$};
\cuboid{1}{2}{2}{8}{0}; \node at (8.5,1) {$...$};
\cuboid{2}{2}{2}{9}{0}; \node at (10,1) {$X_2$}; \end{scope}
\end{scope}
%LAST MATRIX
\begin{scope}[shift = {(21,0.5)}]
\benSquare{3}{2}{0}{0};
\node at (1.5,1) {$A_d$};
\benSquare{3}{1}{0}{2};
\node at (1.5,2.5) {$V_d$};
\node[anchor = east] at (-0.1,2.5) {$\cdot$};
\benSquare{3}{2}{0}{3};
\node at (1.5,4) {$X_d$};
\end{scope}
\end{tikzpicture}

\end{figure}

\begin{proof}
The curve
\[
\gamma:\left(-\varepsilon,\varepsilon\right)\rightarrow\mathcal{M}_{\leq(k_{1}+s_{1},...,k_{d-1}+s_{d-1})}^{n_{1}\times...\times n_{d}}:t\mapsto
\]
\[
\left(\begin{array}{cc}
A_{1}+tX_{1} & U_{1}\end{array}\right)\left(\begin{array}{cc}
A_{2}+tX_{2} & U_{2}\\
tV_{2} & Z_{2}
\end{array}\right)...\left(\begin{array}{cc}
A_{d-1}+tX_{d-1} & U_{d-1}\\
tV_{d-1} & Z_{d-1}
\end{array}\right)\left(\begin{array}{c}
A_{d}+tX_{d}\\
tV_{d}
\end{array}\right)
\]
is analytic and has $\mathcal{X}$ as its first derivative. See this
by differentiating $\gamma$ in $t=0$ using the product rule. For
the basic definition of tangent vector use the sequence $\left(\gamma\left(\frac{1}{m}\right)\right)_{\mathbb{N}\ni m\geq N}$.
\end{proof}
What follows is a technical lemma that facilitates proving both, the
case for order $3$ TT varieties as well as the inductive step for
arbitrary order. Its first two assumptions (equations \ref{eq:firstDecomposition}
and \ref{eq:secondDecomposition}) arrive from applying the matrix
version to the two matricizations with respect to index $1$ and $3$.
The idea of the proof is the following: Represent an arbitrary tangent
vector as the tangent vector of the matricizations using Lemma \ref{lem:T(cup)=00003DCup(T)}.
Then decompose using the result on matrix tangent cones above. Orthogonalizing
with respect to $A_{1}$ and $A_{3}$ allows us to decompose the tangent
vector into an orthogonal sum and compare the orthogonal components
seperately.
\begin{lem}
\label{technical lemma}Let $A\in\mathcal{M}_{=(k_{1},k_{2})}^{n_{1}\times n_{2}\times n_{3}}$
be a singular point in $\mathcal{M}_{\leq(k_{1}+s_{1},k_{2}+s_{2})}^{n_{1}\times n_{2}\times n_{3}}$
($s_{1},s_{2}\geq0$) and let $A_{1}\in\mathbb{R}^{n_{1}\times k_{1}}$,
$A_{2}\in\mathbb{R}^{k_{1}\times n_{2}\times k_{2}}$ and $A_{3}\in\mathbb{R}^{k_{2}\times n_{3}}$
be three tensors such that $A_{1}A_{2}A_{3}=A$. Assume further the
orthogonality of $A_{1}$ and $A_{2}$, $A_{1}^{T}A_{1}=I$, $\left(A_{2}^{R}\right)^{T}A_{2}^{R}=I$.
Let $\mathcal{X}\in\mathbb{R}^{n_{1}\times n_{2}\times n_{3}}$ be
a tensor that admits the $\tilde{s}_{1}$-decomposition
\begin{equation}
\mathcal{X}=A_{1}\mathbf{Y}+XA_{2}A_{3}+U\mathbf{V}\label{eq:firstDecomposition}
\end{equation}
and the $\tilde{s}_{2}$-decomposition
\begin{equation}
\mathcal{X}=A_{1}A_{2}T+\mathbf{S}A_{3}+\mathbf{O}P\label{eq:secondDecomposition}
\end{equation}
with $\tilde{s}_{1}\leq s_{1}$ and $\tilde{s}_{2}\leq s_{2}$. Then
$\mathcal{X}$ is decomposable as
\begin{equation}
\mathcal{X}=\left(\begin{array}{ccc}
A_{1} & U & X\end{array}\right)\left(\begin{array}{ccc}
A_{2} & \dot{O} & \dot{S}\\
0 & Z_{2} & \dot{V}\\
0 & 0 & A_{2}
\end{array}\right)\left(\begin{array}{c}
T\\
P\\
A_{3}
\end{array}\right)\label{eq:star}
\end{equation}
with $\dot{O}=A_{1}^{\dagger}\mathbf{O}$, $\dot{S}=A_{1}^{\dagger}\mathbf{S}$
and $\dot{V}=\mathbf{V}A_{3}^{\dagger}$. In particular we have the
orthogonality statements $\left(A_{2}^{R}\right)^{T}\dot{O}^{R}=0$,
$\left(A_{2}^{R}\right)^{T}\dot{S}^{R}=0$, $\left(\dot{V}A_{3}\right)^{L}\left(\left(A_{2}A_{3}\right)^{L}\right)^{T}=0$
and that $Z_{2}P+\dot{V}A_{3}$ and $A_{1}\dot{O}+UZ_{2}$ have full
rank and the equivalence

\[
\left(\begin{array}{ccc}
A_{1} & U & X\end{array}\right)\left(\begin{array}{ccc}
A_{2} & \dot{O} & \dot{S}\\
0 & Z_{2} & \dot{V}\\
0 & 0 & A_{2}
\end{array}\right)=\left(\begin{array}{ccc}
A_{1}A_{2} & \mathbf{S} & \mathbf{O}\end{array}\right).
\]
\end{lem}
\begin{proof}
Define $\dot{Y}:=\mathbf{Y}A_{3}^{\dagger}$, $\dot{V}:=\mathbf{V}A_{3}^{\dagger}$,
$\dot{T}:=TA_{3}^{\dagger}$, $\dot{S}:=A_{1}^{\dagger}\mathbf{S}$
and $\dot{O}:=A_{1}^{\dagger}\mathbf{O}$, where $A^{\dagger}$ denotes
the Moore-Penrose Pseudoinverse and can be replaced by $A^{T}$ for
orthogonal matrices and by $A^{T}(AA^{T})^{-1}$ for full rank matrices
with more columns than rows. Then we can decompose $\mathbf{Y}$,
$\mathbf{V}$, $\mathbf{S}$, $\mathbf{O}$ and $T$ into

\[
\mathbf{Y}=\hat{Y}+\dot{Y}A_{3}\text{, }\mathbf{V}=\hat{V}+\dot{V}A_{3}\text{, }\mathbf{O}=\hat{O}+A_{1}\dot{O}\text{, }\mathbf{S}=\hat{S}+A_{1}\dot{S}\text{, }T=\hat{T}+\dot{T}A_{3}
\]
where the hat-wearing variables are orthogonal to $A_{1}$ or $A_{3}$
respectively: 
\[
\hat{Y}A_{3}^{T}=0,\,\hat{V}A_{3}^{T}=0,\,A_{1}^{T}\hat{O}=0,\,A_{1}^{T}\hat{S}=0,\,\hat{T}A_{3}^{T}=0.
\]
Then we can write the tangent vector as an orthogonal sum (w.r.t.
the scalar product on $\mathbb{R}^{n_{1}n_{2}n_{3}}$) in the four
spaces
\[
\operatorname{range}(A_{1})\otimes\mathbb{R}^{n_{2}}\otimes\operatorname{range}(A_{3}^{T}),
\]
\[
\operatorname{range}(A_{1})\otimes\mathbb{R}^{n_{2}}\otimes\operatorname{range}(A_{3}^{T})^{\perp},
\]
\[
\operatorname{range}(A_{1})^{\perp}\otimes\mathbb{R}^{n_{2}}\otimes\operatorname{range}(A_{3}^{T}),
\]
\[
\operatorname{range}(A_{1})^{\perp}\otimes\mathbb{R}^{n_{2}}\otimes\operatorname{range}(A_{3}^{T})^{\perp}.
\]
Rewriting equations \ref{eq:firstDecomposition} and \ref{eq:secondDecomposition}
yields

\begin{equation}
\mathcal{X}=A_{1}\dot{Y}A_{3}+A_{1}\hat{Y}+(XA_{2}+U\dot{V})A_{3}+U\hat{V}.\label{eq:1}
\end{equation}
and
\begin{equation}
\mathcal{X}=A_{1}(A_{2}\dot{T}+\dot{S})A_{3}+A_{1}(A_{2}\hat{T}+\dot{O}P)+\hat{S}A_{3}+\hat{O}P.\label{eq:2}
\end{equation}
respectively. Both representations need to be equal. Because they
are orthogonal sums in the same four spaces, each summand has to be
equal to the corresponding summand in the other sum. In particular
we have
\[
\hat{O}P=U\hat{V}.
\]
By defining $Z_{2}:=U^{\dagger}\hat{O}$, we can write
\begin{equation}
U\hat{V}=\hat{O}P=UZ_{2}P\label{eq:3}
\end{equation}
and see that $Z_{2}=\hat{V}P^{\dagger}$ (by multiplying equation
\ref{eq:3} by the full rank matrices $U^{\dagger}$ and $P^{\dagger}$).
Using the first and second summand of equation \ref{eq:2}, the third
summand of equation \ref{eq:1} and equation \ref{eq:3} we assemble
the desired representation from equation \ref{eq:star}

\[
\mathcal{X}=A_{1}\dot{S}A_{3}+A_{1}A_{2}T+A_{1}\dot{O}P+XA_{2}A_{3}+U\dot{V}A_{3}+UZ_{2}P.
\]
with all the desired properties. See this in the following way: $A_{1}\dot{O}+UZ_{2}=A_{1}\dot{O}+\hat{O}=\mathbf{O}$
is orthogonal to $A_{1}A_{2}$, therefore $0=\left(\left(A_{1}A_{2}\right)^{R}\right)^{T}\mathbf{O}^{R}=\left(\left(A_{1}A_{2}\right)^{R}\right)^{T}\left(A_{1}\dot{O}+UZ_{2}\right)^{R}=\left(\left(A_{1}A_{2}\right)^{R}\right)^{T}\left(A_{1}\dot{O}\right)^{R}=\left(A_{2}^{R}\right)^{T}\dot{O}^{R}$.
And analogously for $Z_{2}P+\dot{V}A_{3}=\mathbf{V}$ and $A_{1}\dot{S}+U\dot{V}+XA_{2}=\mathbf{S}$
(by $XA_{2}+U\dot{V}=\hat{S}$ from equations \ref{eq:1} and \ref{eq:2}).
\end{proof}
We can now state our main result for arbitrary TT varieties.
\begin{thm}
\label{thm:arbitrary_d}Let $A\in\mathcal{M}_{=(k_{1},...,k_{d-1})}^{n_{1}\times...\times n_{d}}$
be a singular point in $\mathcal{M}_{\leq(k_{1}+s_{1},...,k_{d-1}+s_{d-1})}^{n_{1}\times...\times n_{d}}$
($s_{i}\geq0$) and let $A_{1}\in\mathbb{R}^{n_{1}\times k_{1}}$,
$A_{2}\in\mathbb{R}^{k_{1}\times n_{2}\times k_{2}}$,... and $A_{d}\in\mathbb{R}^{k_{d-1}\times n_{d}}$
be tensors such that $A_{1}...A_{d}=A$ and $A_{1}^{T}A_{1}=I$, $\left(A_{i}^{R}\right)^{T}A_{i}^{R}=I\;\forall i=2,...,d-1$.
Then any vector in the tangent cone of $\mathcal{M}_{\leq(k_{1}+s_{1},...,k_{d-1}+s_{d-1})}^{n_{1}\times...\times n_{d}}$
at the point $A_{1}...A_{d}$ can be written as the TT tensor

\[
\left(\begin{array}{ccc}
A_{1} & U_{1} & X_{1}\end{array}\right)\left(\begin{array}{ccc}
A_{2} & U_{2} & X_{2}\\
0 & Z_{2} & V_{2}\\
0 & 0 & A_{2}
\end{array}\right)...\left(\begin{array}{ccc}
A_{d-1} & U_{d-1} & X_{d-1}\\
0 & Z_{d-1} & V_{d-1}\\
0 & 0 & A_{d-1}
\end{array}\right)\left(\begin{array}{c}
X_{d}\\
V_{d}\\
A_{d}
\end{array}\right)
\]
where $\left(A_{i}^{R}\right)^{T}U_{i}^{R}=0\;\forall i$, $\left(A_{i}^{R}\right)^{T}X_{i}^{R}=0\;\forall i\neq d$,
$\left(V_{i}A_{i+1}...A_{d}\right)^{L}\left(\left(A_{i}...A_{d}\right)^{L}\right)^{T}=0\;\forall i$.
\end{thm}
\begin{proof}
The idea of the proof is illustrated in figure \ref{fig:proof-idea}.
\begin{figure}
\caption{\label{fig:proof-idea}proof of the theorem}
\begin{tikzpicture}[scale=0.5]
\draw[black,fill=black] (0,0) circle (1ex); \node[anchor=south] at (0,0) {$A_2$};
\draw[black,fill=black] (3,0) circle (1ex); \node[anchor=south] at (3,0) {$A_d$};
\draw[black,fill=black] (-2,0) circle (1ex); \node[anchor=south] at (-2,0) {$A_1$};
\draw(-2,-1)--(-2,0)--(0,0)--(1,0); \draw(2,0)--(3,0)--(3,-1); \node[anchor = west] at (1,0) {...};
\draw(0,0)--(0,-1);
\begin{scope}[shift = {(5,0)}] \node[] at (0,0) {$\longrightarrow{}$}; \node[anchor=south] at (0,0) {\tiny contract}; \end{scope}
\begin{scope}[shift = {(9,0)}] \draw[black,fill=black] (0,0) circle (1ex); \node[anchor=south] at (0,0) {$A_2$};
\draw[black,fill=black] (2,0) circle (1ex); \node[anchor=south] at (2,0) {$A_3...A_d$};
\draw[black,fill=black] (-2,0) circle (1ex); \node[anchor=south] at (-2,0) {$A_1$};
\draw(-2,-1)--(-2,0)--(0,0)--(1,0); \draw(1,0)--(2,0)--(2,-1); \draw(2,0)--(3,0); \draw(2,0)--(2.7,-0.7);
\draw(0,0)--(0,-1); \end{scope}
\begin{scope}[shift = {(13.5,0)}] \node[] at (0,0) {$\longrightarrow{}$}; \node[anchor=south] at (0,0) {\tiny matricize}; \end{scope}
\begin{scope}[shift = {(17.3,0)}] \draw[black,fill=black] (0,0) circle (1ex); \node[anchor=south] at (0,0) {$A_2$};
\draw[black,fill=black] (2,0) circle (1ex); \node[anchor=south] at (2,0) {$A_3...A_d$};
\draw[black,fill=black] (-2,0) circle (1ex); \node[anchor=south] at (-2,0) {$A_1$};
\draw(-2,-1)--(-2,0)--(0,0)--(1,0); \draw(1,0)--(2,0); \draw[line width=2pt](2,0)--(2,-1);
\draw(0,0)--(0,-1); \end{scope}
\begin{scope}[shift = {(17.4,-1.8)}] \node[] at (0,0) {$\downarrow{}$}; \node[anchor=west] at (0,0) {\tiny Lemma \ref{technical lemma}}; \end{scope}
\begin{scope}[shift = {(17.3,-4)}]
\node[anchor=south] at (-3,-0.5) {$\tiny \left(\begin{array}{ccc} A_{1} & U & X\end{array}\right)\left(\begin{array}{ccc} A_{2} & \dot{O} & \dot{S}\\ 0 & Z_{2} & \dot{V}\\ 0 & 0 & A_{2} \end{array}\right)\left(\begin{array}{c} T\\ P\\ A_{3}...A_{d} \end{array}\right)$};
\node[anchor=south] at (-1.17,-2.9) {$\tiny =\left(\begin{array}{ccc} A_{1}A_{2} & \mathbf{S} & \mathbf{O}\end{array}\right)\left(\begin{array}{c} T\\ P\\ A_{3}...A_{d} \end{array}\right)$}; \end{scope}
\node[] at (0.1,-1.8) {$\downarrow{}$}; \node[anchor=west] at (0.1,-1.8) {\tiny contract and matricize};
\begin{scope}[shift = {(2,-3.2)}] \draw[black,fill=black] (0,0) circle (1ex); \node[anchor=south west] at (0,0) {$A_4...A_d$};
\draw[black,fill=black] (-2,0) circle (1ex); \node[anchor=south east] at (-2,0) {$A_1 A_2 A_3$};
\draw(-2,0)--(0,0); \draw[line width=2pt](0,0)--(0,-1); \draw[line width=2pt](-2,-1)--(-2,0); \end{scope}
\node[] at (0.1,-5) {$\downarrow{}$}; \node[anchor=west] at (0.1,-5) {\tiny Harris, Absil, Uschmajew};
\node[] at (2,-6) {$\tiny\left(\begin{array}{ccc} A_{1}A_{2}A_{3} & \mathbf{S}_2 & \mathbf{O}_2\end{array}\right)\left(\begin{array}{c} T_2\\ P_2\\ A_{4}...A_{d} \end{array}\right)$};
\draw [decorate,decoration={brace,amplitude=10pt},xshift=-4pt,yshift=0pt] (20,-7) -- (-2,-7) node [black,midway,yshift=-0.6cm]  {\footnotesize Lemma \ref{technical lemma} ... Induction};
\end{tikzpicture}
\end{figure}
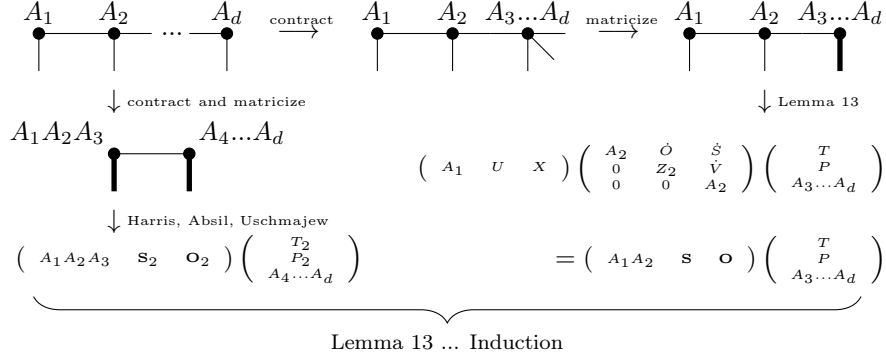
Applying the matrix version of this theorem \cite[Thm 3.2]{absil_graph_similarity,uschmajew_schneider_matrix_case}
to the matricizations from $\mathcal{M}_{\leq(k_{1}+s_{1})}^{n_{1}\times(n_{2}...n_{d})}$
and to $\mathcal{M}_{\leq(k_{2}+s_{2})}^{(n_{1}n_{2})\times(n_{3}...n_{d-2})}$
we arrive at the assumptions of Lemma \ref{technical lemma} and can
decompose the tangent vector in the form
\[
\mathcal{X}=\left(\begin{array}{ccc}
A_{1} & U_{1} & X_{1}\end{array}\right)\left(\begin{array}{ccc}
A_{2} & U_{2} & X_{2}\\
0 & Z_{2} & V_{2}\\
0 & 0 & A_{2}
\end{array}\right)\left(\begin{array}{c}
T_{3}\\
P_{3}\\
A_{3}...A_{d}
\end{array}\right)
\]
with $U_{1}$ and $X_{1}$ orthogonal to $A_{1}$, the two matrices
$U_{2}$ and $X_{2}$ orthogonal to $A_{1}$ and $(V_{2}A_{3})^{L}$
orthogonal to $(A_{2}A_{3})^{L}$ from the left and right respectively
and $A_{1}U_{2}+U_{1}Z_{2}$ having full rank. Using this as inductive
basis we continue by proving the inductive step: Assume that $\mathcal{X}$
has the decomposition
\[
\mathcal{X}=\left(\begin{array}{ccc}
A_{1} & U_{1} & X_{1}\end{array}\right)\left(\begin{array}{ccc}
A_{2} & U_{2} & X_{2}\\
0 & Z_{2} & V_{2}\\
0 & 0 & A_{2}
\end{array}\right)...\left(\begin{array}{ccc}
A_{i} & U_{i} & X_{i}\\
0 & Z_{i} & V_{i}\\
0 & 0 & A_{i}
\end{array}\right)\left(\begin{array}{c}
T_{i+1}\\
P_{i+1}\\
A_{i+1}...A_{d}
\end{array}\right)
\]
with $\left(A_{i}^{R}\right)^{T}U_{i}^{R}=0\;\forall i$, $\left(A_{i}^{R}\right)^{T}X_{i}^{R}=0\;\forall i$,
$\left(V_{i}A_{i+1}...A_{d}\right)^{L}\left(\left(A_{i}...A_{d}\right)^{L}\right)^{T}=0\;\forall i$.
Then we see that in the contraction
\begin{equation}
\left(\begin{array}{ccc}
A_{1}...A_{i} & B & C\end{array}\right):=\left(\begin{array}{ccc}
A_{1} & U_{1} & X_{1}\end{array}\right)\left(\begin{array}{ccc}
A_{2} & U_{2} & X_{2}\\
0 & Z_{2} & V_{2}\\
0 & 0 & A_{2}
\end{array}\right)...\left(\begin{array}{ccc}
A_{i} & U_{i} & X_{i}\\
0 & Z_{i} & V_{i}\\
0 & 0 & A_{i}
\end{array}\right)\label{eq:BigStar}
\end{equation}
$B^{R}$ and $C^{R}$ are both orthogonal to $\left(A_{1}...A_{i}\right)^{R}$
from the left, by which we mean $\left(\left(A_{1}...A_{i}\right)^{R}\right)^{T}B^{R}=0$
and $\left(\left(A_{1}...A_{i}\right)^{R}\right)^{T}C^{R}=0$. We
thus have the first assumption (equation \ref{eq:firstDecomposition})
of Lemma \ref{technical lemma} for the variety $\mathcal{M}_{\leq(k_{i}+s_{i},k_{i+1}+s_{i+1})}^{(n_{1}...n_{i})\times n_{i+1}\times(n_{i+2}...n_{d})}$.
The second assumption follows by the matrix version from \cite{uschmajew_schneider_matrix_case}.
Thus we can apply Lemma \ref{technical lemma} to achieve the decomposition
\[
\mathcal{X}=\left(\begin{array}{ccc}
A_{1}...A_{i} & B & C\end{array}\right)\left(\begin{array}{ccc}
A_{i+1} & U_{i+1} & X_{i+1}\\
0 & Z_{i+1} & V_{i+1}\\
0 & 0 & A_{i+1}
\end{array}\right)\left(\begin{array}{c}
T_{i+2}\\
P_{i+2}\\
A_{i+2}...A_{d}
\end{array}\right)
\]
Combining this with equation \ref{eq:BigStar} completes the inductive
step and the proof of Theorem \ref{thm:arbitrary_d}.
\end{proof}
\begin{rem}
For parametrizing the tangent cone, we use the same number of parameters
as in the parametrizations of the TT variety. Each block $\left(\begin{array}{cc}
U_{i} & X_{i}\\
Z_{i} & V_{i}
\end{array}\right)$ is of size $(k_{i-1}+s_{i-1})\times(k_{i}+s_{i})$.
\end{rem}
~
\begin{rem}
\label{rem:orthogonal_sum_of_many_terms}Evaluating the expression
\[
\left(\begin{array}{ccc}
A_{1} & U_{1} & X_{1}\end{array}\right)\left(\begin{array}{ccc}
A_{2} & U_{2} & X_{2}\\
0 & Z_{2} & V_{2}\\
0 & 0 & A_{2}
\end{array}\right)...\left(\begin{array}{ccc}
A_{d-1} & U_{d-1} & X_{d-1}\\
0 & Z_{d-1} & V_{d-1}\\
0 & 0 & A_{d-1}
\end{array}\right)\left(\begin{array}{c}
X_{d}\\
V_{d}\\
A_{d}
\end{array}\right)
\]
for the tangent cone parametrization yields
\[
A_{1}...A_{d-1}X_{d}+A_{1}...A_{d-2}X_{d-1}A_{d}+...+X_{1}A_{2}...A_{d}
\]
\[
+A_{1}...A_{d-2}U_{d-1}V_{d}+A_{1}...A_{d-3}U_{d-2}V_{d-1}A_{d}+...+U_{1}V_{2}A_{3}...A_{d}
\]
\[
+A_{1}...A_{d-3}U_{d-2}Z_{d-1}V_{d}+...+U_{1}Z_{2}V_{3}A_{4}...A_{d}
\]
\[
\vdots
\]
\[
+U_{1}Z_{2}...Z_{d-1}V_{d}
\]
where all summands are pairwise orthogonal in the standard scalar
product on $\mathbb{R}^{n_{1}...n_{d}}$. Note that an ALS algorithm
only uses directions from the first line of this decomposition. The
DMRG algorithm additionally uses directions from the second line.
See \cite{ALS_DMRG_Schneider} for a study of both, ALS and DMRG.
\end{rem}
We can deduce, that in the case of TT varieties the intersection of
the tangent cones is the tangent cone of the intersection.
\begin{cor}
\label{cor:IntersectionOfTangentCones}
\[
\bigcap_{i=1,...,d-1}T_{A}\mathcal{M}_{\leq k_{i}}^{(n_{1}...n_{i})\times(n_{i+1}...n_{d})}\subset T_{A}\mathcal{M}_{\leq(k_{1},...,k_{d-1})}^{n_{1}\times...\times n_{d}}
\]
 and thus
\[
T_{A}\mathcal{M}_{\leq(k_{1},...,k_{d-1})}^{n_{1}\times...\times n_{d}}=\bigcap_{i=1,...,d-1}T_{A}\mathcal{M}_{\leq k_{i}}^{(n_{1}...n_{i})\times(n_{i+1}...n_{d})}
\]
\end{cor}
\begin{proof}
If 
\[
\mathcal{X}\in\bigcap_{i=1,...,d-1}T_{A}\mathcal{M}_{\leq k_{i}}^{(n_{1}...n_{i})\times(n_{i+1}...n_{d})}
\]
 then by Lemma \ref{technical lemma} works and we can find coefficient
tensors such that we can write $\mathcal{X}$ in our parametrization.
But then by Lemma \ref{lem:subset_of_tangent_cone}
\[
\mathcal{X}\in T_{A}\mathcal{M}_{\leq(k_{1},...,k_{d-1})}^{n_{1}\times...\times n_{d}}.
\]
\end{proof}
This corollary was unexpected because of the following example.
\begin{example}
The tangent cone of the intersection is not always equal to the intersection
of the tangent cones. Consider the plane $\mathcal{M}:=\{(x,y,z)\in\mathbb{R}^{3}:\,x=0\}$
and the cylinder $\mathcal{N}:=\{(x,y,z)\in\mathbb{R}^{3}:(x-1)^{2}+y^{2}=1\}$
and the point $(0,0,0)\in\mathcal{N}\cap\mathcal{M}$. Being the line
where both varieties touch, the tangent cone $T_{A}\mathcal{M}$ of
$\mathcal{M}$ at $A$ is the same as the tangent cone of $\mathcal{N}$
at $A$, namely the $y$-$z$-plane. However the tangent cone of $\mathcal{M}\cap\mathcal{N}=\{(x,y,z)\in\mathbb{R}^{3}:x=y=0\}$
at $A$ is only the $z$-axis.
\end{example}
We can show that the issue raised in example \ref{exa:analytic curves}
is unimportant for TT varieties. Namely:
\begin{cor}
\label{cor:first_der_suffice}The tangent cone to a TT variety is
equivalent to the set of all first derivatives to analytic arcs.
\end{cor}
\begin{proof}
By theorem \ref{thm:arbitrary_d} every tangent vector can be written
in the form
\[
\left(\begin{array}{ccc}
A_{1} & U_{1} & X_{1}\end{array}\right)\left(\begin{array}{ccc}
A_{2} & U_{2} & X_{2}\\
0 & Z_{2} & V_{2}\\
0 & 0 & A_{2}
\end{array}\right)...\left(\begin{array}{ccc}
A_{d-1} & U_{d-1} & X_{d-1}\\
0 & Z_{d-1} & V_{d-1}\\
0 & 0 & A_{d-1}
\end{array}\right)\left(\begin{array}{c}
X_{d}\\
V_{d}\\
A_{d}
\end{array}\right)
\]
 and by Lemma \ref{lem:subset_of_tangent_cone} this is the first
derivative of the analytic curve
\[
\gamma:\,t\mapsto
\]
\[
\left(\begin{array}{cc}
A_{1}+tX_{1} & U_{1}\end{array}\right)\left(\begin{array}{cc}
A_{2}+tX_{2} & U_{2}\\
tV_{2} & Z_{2}
\end{array}\right)...\left(\begin{array}{cc}
A_{d-1}+tX_{d-1} & U_{d-1}\\
tV_{d-1} & Z_{d-1}
\end{array}\right)\left(\begin{array}{c}
A_{d}+tX_{d}\\
tV_{d}
\end{array}\right).
\]
The converse is trivial by using the sequence $\left(\eta\left(\frac{1}{m}\right)\right)_{\mathbb{N}\ni m\geq N}$
for an analytic curve $\eta$.
\end{proof}

\section{Retraction onto the variety}

As a retraction one can use the curve from Lemma \ref{lem:subset_of_tangent_cone}.
For a retraction we adopt the definition from \cite{uschmajew_schneider_matrix_case}:
\begin{defn}
Let $\mathcal{M}$ be an algebraic variety. The tangent bundle of
the variety $\mathcal{M}$ is the set $\bigcup_{x\in\mathcal{M}}\left(\{x\}\times T_{x}\mathcal{M}\right)$.
A retraction is a function $R$ from the tangent bundle to the variety
such that for any fixed $x\in\mathcal{M}$ and $v\in T_{x}\mathcal{M}$
the function $t\mapsto R(x,tv)$ is continuous on $[0,\infty)$ and
\[
\lim_{t\searrow0}\frac{R(x,tv)-x-tv}{t}=0.
\]
\end{defn}
\begin{lem}
The function
\[
R:\mathcal{X}=\left(\begin{array}{ccc}
A_{1} & U_{1} & X_{1}\end{array}\right)\left(\begin{array}{ccc}
A_{2} & U_{2} & X_{2}\\
0 & Z_{2} & V_{2}\\
0 & 0 & A_{2}
\end{array}\right)...\left(\begin{array}{ccc}
A_{d-1} & U_{d-1} & X_{d-1}\\
0 & Z_{d-1} & V_{d-1}\\
0 & 0 & A_{d-1}
\end{array}\right)\left(\begin{array}{c}
X_{d}\\
V_{d}\\
A_{d}
\end{array}\right)
\]
\[
\mapsto\left(\begin{array}{cc}
A_{1}+X_{1} & U_{1}\end{array}\right)\left(\begin{array}{cc}
A_{2}+X_{2} & U_{2}\\
V_{2} & Z_{2}
\end{array}\right)...\left(\begin{array}{cc}
A_{d-1}+X_{d-1} & U_{d-1}\\
V_{d-1} & Z_{d-1}
\end{array}\right)\left(\begin{array}{c}
A_{d}+X_{d}\\
V_{d}
\end{array}\right)
\]
defines a retraction in the sense of the definition above.
\end{lem}
\begin{proof}
The image under $R$ of the tangent vector multiplied by $t$, $R(t\mathcal{X})$
is

\[
\left(\begin{array}{cc}
A_{1}+tX_{1} & U_{1}\end{array}\right)\left(\begin{array}{cc}
A_{2}+tX_{2} & U_{2}\\
tV_{2} & Z_{2}
\end{array}\right)...\left(\begin{array}{cc}
A_{d-1}+tX_{d-1} & U_{d-1}\\
tV_{d-1} & Z_{d-1}
\end{array}\right)\left(\begin{array}{c}
A_{d}+tX_{d}\\
tV_{d}
\end{array}\right).
\]
We calculate
\[
\lim_{t\searrow0}\frac{R(x,tv)-x-tv}{t}=\lim_{t\searrow0}\frac{t^{2}(\mbox{polynomial in \ensuremath{t}})}{t}=\lim_{t\searrow0}t(\mbox{polynomial in \ensuremath{t}})=0
\]
\end{proof}
Note that this retraction is particularly easy to calculate if the
tangent vectors are given in the described format.

\section{The hierarchical format}

All of the above generalises in a straight-forward way to the hierarchical
and Tucker format. However the notation is difficult. Therefore we
will omit some details. See \cite{hackbusch} or \cite{kuehn} for
the definition and a detailed study of the hierarchical tensor format.
We will only give the equivalent of the technical Lemma \ref{technical lemma}
for the Tucker format with order $3$. This will allow us to use the
same inductive step as in theorem \ref{thm:arbitrary_d} to prove
the parametrization for any binary tree. In further generalizing the
technical lemma to arbitrary Tucker formats, one could prove the theorem
for arbitrary tree formats.

Let $A_{1}\in\mathbb{R}^{n_{1}\times k_{1}}$, $A_{2}\in\mathbb{R}^{k_{2}\times n_{2}}$
, $A_{3}\in\mathbb{R}^{n_{3}\times k_{3}}$ and $A_{4}\in\mathbb{R}^{k_{1}\times k_{2}\times k_{3}}$
with $A_{1}$, $A_{2}$ and $A_{3}$ having full rank. For writing
simple tensor tree diagrams, we can use the Kronecker product. Sorting
the indices $k_{1}$ and $k_{3}$ lexicographically, we can identify
the tree diagram and the term depicted in figure \ref{fig:Kronecker-product-notation}.

\begin{figure}[h]
\caption{\label{fig:Kronecker-product-notation}Kronecker product notation
for tensor trees}
\hspace{1.5cm}\begin{tikzpicture}[scale=0.5]
\draw[black,fill=black] (0,0) circle (1ex); \node[anchor=south] at (0,0) {$A_4$};
\draw[black,fill=black] (2,0) circle (1ex); \node[anchor=south] at (2,0) {$A_3$};
\draw[black,fill=black] (-2,0) circle (1ex); \node[anchor=south] at (-2,0) {$A_1$};
\draw[black,fill=black] (0,-2) circle (1ex); \node[anchor=east] at (0,-2) {$A_2$};
\node[anchor=south] at (-1,0) {\tiny $k_1$}; \node[anchor=south] at (1,0) {\tiny $k_3$}; \node[anchor=east] at (0,-1) {\tiny $k_2$};
\draw(-3,0)--(-2,0)--(0,0)--(2,0)--(3,0);
\draw(0,0)--(0,-3);
\node[anchor = west] at (4,-1) {$=\left((A_1\otimes A_3)A_4^{(k_1k_3)\times k_2}A_2\right)^{n_1\times n_2 \times n_3}$};
\end{tikzpicture}

\end{figure}

We can write this in the following three ways:

\[
\left((A_{1}\otimes A_{3})A_{4}^{(k_{1}k_{3})\times k_{2}}\right)\cdot A_{2}
\]
\[
=A_{1}\cdot\left(A_{4}^{k_{1}\times(k_{3}k_{2})}(A_{3}\otimes A_{2})\right)
\]
\[
=A_{3}\cdot\left(A_{4}^{k_{3}\times(k_{1}k_{2})}(A_{1}\otimes A_{2})\right)
\]
Now any tangent vector from a tucker variety $\mathcal{M}_{\leq(k_{1}+s_{1},k_{2}+s_{2},k_{3}+s_{3})}^{n_{1}\times n_{2}\times n_{3}}$
(we use the obvious generalization of the symbols defined for the
TT varieties) parametrized by $A_{1}$, $A_{2}$, $A_{3}$ and $A_{4}$
can be decomposed in the $\tilde{s}_{2}$-decomposition
\[
\left((A_{1}\otimes A_{3})A_{4}^{(k_{1}k_{3})\times k_{2}}\right)Y_{2}+\mathbf{X}_{2}A_{2}+\mathbf{U}_{2}V_{2},
\]
in the $\tilde{s}_{1}$-decomposition
\[
A_{1}\mathbf{Y}_{1}+X_{1}\left(A_{4}^{k_{1}\times(k_{2}k_{3})}(A_{2}\otimes A_{3})\right)+U_{1}\mathbf{V_{1}}
\]
and the $\tilde{s}_{3}$-decomposition
\[
A_{3}\mathbf{Y}_{3}+X_{3}\left(A_{4}^{k_{3}\times(k_{1}k_{2})}(A_{1}\otimes A_{2})\right)+U_{3}\mathbf{V}_{3}
\]
with $\tilde{s}_{1}\leq s_{1}$, $\tilde{s}_{2}\leq s_{2}$ and $\tilde{s}_{3}\leq s_{3}$.
We can further decompose each of the three into the $8$ orthogonal
subspaces
\[
\operatorname{range}(A_{1})\otimes\operatorname{range}(A_{2}^{T})\otimes\operatorname{range}(A_{3}),\quad\operatorname{range}(A_{1})\otimes\operatorname{range}(A_{2}^{T})\otimes\operatorname{range}(A_{3})^{\perp},
\]
\[
\operatorname{range}(A_{1})^{\perp}\otimes\operatorname{range}(A_{2}^{T})\otimes\operatorname{range}(A_{3}),\quad\operatorname{range}(A_{1})^{\perp}\otimes\operatorname{range}(A_{2}^{T})\otimes\operatorname{range}(A_{3})^{\perp},
\]
\[
\operatorname{range}(A_{1})\otimes\operatorname{range}(A_{2}^{T})^{\perp}\otimes\operatorname{range}(A_{3}),\quad\operatorname{range}(A_{1})\otimes\operatorname{range}(A_{2}^{T})^{\perp}\otimes\operatorname{range}(A_{3})^{\perp},
\]
\[
\operatorname{range}(A_{1})^{\perp}\otimes\operatorname{range}(A_{2}^{T})^{\perp}\otimes\operatorname{range}(A_{3}),\quad\operatorname{range}(A_{1})^{\perp}\otimes\operatorname{range}(A_{2}^{T})^{\perp}\otimes\operatorname{range}(A_{3})^{\perp},
\]
Exemplarily we further decompose the $\tilde{s}_{1}$-decomposition.
For this purpose we need to write $\mathbf{Y}_{1}$ as the orthogonal
sum
\[
\mathbf{Y}_{1}^{k_{1}\times(n_{2}n_{3})}=(I\otimes A_{3})\dot{Y}_{1}A_{2}+Y_{1}^{3}A_{2}+(I\otimes A_{3})Y_{1}^{2}+Y_{1}^{2,3}
\]
such that $(I\otimes A_{3})^{T}Y_{1}^{3}=0$, $Y_{1}^{2}A_{2}^{T}=0$,
$(I\otimes A_{3})^{T}Y_{1}^{2,3}=0$ and $Y_{1}^{2,3}A_{2}^{T}=0$
(use pseudo inverses for this purpose as in equation \ref{eq:pseudoinverses}).
Analogously we rewrite $\mathbf{V}_{1}$ as
\[
\mathbf{V}_{1}^{k_{1}\times(n_{2}n_{3})}=(I\otimes A_{3})\dot{V}_{1}A_{2}+V_{1}^{3}A_{2}+(I\otimes A_{3})V_{1}^{2}+V_{1}^{2,3}
\]
 such that the $\tilde{s}_{1}$-decomposition can be rewritten as
the orthogonal sum
\begin{equation}
(A_{1}\otimes A_{3})\dot{Y}_{1}A_{2}\label{eq:dec1}
\end{equation}
\begin{equation}
+(A_{1}\otimes I)Y_{1}^{3}A_{2}\label{eq:dec1-1}
\end{equation}
\begin{equation}
+(A_{1}\otimes A_{3})Y_{1}^{2}\label{eq:dec1-2}
\end{equation}
\begin{equation}
+(A_{1}\otimes I)Y_{1}^{2,3}\label{eq:dec1-3}
\end{equation}
\begin{equation}
+((U_{1}\otimes A_{3})\dot{V}_{1}+(X_{1}\otimes A_{3})A_{4})A_{2}\label{eq:dec1-4}
\end{equation}
\begin{equation}
+U_{1}V_{1}^{3}A_{2}\label{eq:dec1-5}
\end{equation}
\begin{equation}
+(U_{1}\otimes A_{3})V_{1}^{2}\label{eq:dec1-6}
\end{equation}
\begin{equation}
+U_{1}V_{1}^{2,3}\label{eq:dec1-7}
\end{equation}
Comparing coefficients with the orthogonal decompositions of the $\tilde{s}_{2}$-
and $\tilde{s}_{3}$-decompositions, we arrive at the representation
\[
\mathcal{X}=\left((\begin{array}{ccc}
A_{1} & U_{1} & X_{1}\end{array})\otimes(\begin{array}{ccc}
A_{3} & U_{3} & X_{3}\end{array})\right)\mathbf{C}\left(\begin{array}{c}
Y_{2}\\
V_{2}\\
A_{2}
\end{array}\right)
\]
with $\mathbf{C}\in\mathbb{R}^{(k_{1}+\tilde{s}_{1}+k_{1})(k_{3}+\tilde{s}_{3}+k_{3})\times(k_{2}+\tilde{s}_{2}+k_{2})}$
having the form depicted in figure \ref{fig:3d-diagram}.

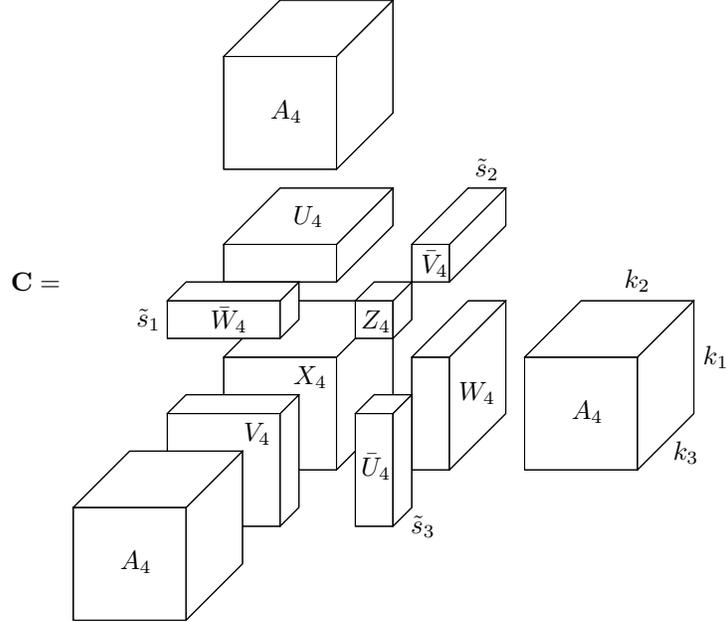
\begin{figure}[H]
\caption{\label{fig:3d-diagram}Coefficient tensor of tangent cone parametrization
for order 3 Tucker}
\hspace{1.5cm}\begin{tikzpicture}[scale=0.5]

\node[] at (-5,5){$\mathbf{C}=$};

\newcommand{\cuboid}[6]{
\begin{scope}[shift={(#4,#5)}]
\draw[fill=white](#1,#2)--(#1,0)--(0,0)--(0,#2)--cycle;
\draw[fill=white](#1,0)--(#1+#3/2,#3/2)--(#1+#3/2,#2+#3/2)--(#1,#2)--cycle;
\draw[fill=white](#1+#3/2,#2+#3/2)--(#3/2,#2+#3/2)--(0,#2)--(#1,#2)--cycle;
\draw[draw=black](#1,#2)--(#1,0)--(0,0)--(0,#2)--(#1,#2)--(#1+#3/2,#2+#3/2);
\draw[draw=black](#1,0)--(#1+#3/2,#3/2)--(#1+#3/2,#2+#3/2)--(#3/2,#2+#3/2)--(0,#2);
\end{scope}
}
\cuboid{3}{3}{3}{0}{0};
\node[anchor=north east] at (3,3) {$X_4$};
\cuboid{3}{1}{3}{0}{5};
\node[anchor=south west] at (1.6,6.2) {$U_4$};
\cuboid{3}{3}{3}{0}{8};
\node[anchor=south west] at (1,9) {$A_4$};
\cuboid{3}{3}{1}{-1.5}{-1.5};
\node[anchor=north east] at (1.5,1.5) {$V_4$};
\cuboid{3}{1}{1}{-1.5}{3.5};
\node[] at (-2,4) {$\tilde{s}_1$};
\node[anchor=south west] at (-0.6,3.4) {$\bar{W}_4$};
\cuboid{3}{3}{3}{-4}{-4};
\node[anchor=south west] at (-3,-3) {$A_4$};
\cuboid{1}{3}{3}{5}{0};
\node[] at (7,8) {$\tilde{s}_2$};
\node[anchor=south west] at (6,1.5) {$W_4$};
\cuboid{3}{3}{3}{8}{0};
\node[anchor=south west] at (9,1) {$A_4$};
\node[anchor=south] at (11,4.5) {$k_2$};
\node[anchor=west] at (12.5,3) {$k_1$};
\node[anchor=north west] at (11.7,1) {$k_3$};
\cuboid{1}{1}{3}{5}{5};
\node[anchor=south west] at (5,4.9) {$\bar{V}_4$};
\cuboid{1}{3}{1}{3.5}{-1.5};
\node[] at (5.3,-1.5) {$\tilde{s}_3$};
\node[anchor=south west] at (3.4,-0.5) {$\bar{U}_4$};
\cuboid{1}{1}{1}{3.5}{3.5};
\node[anchor=south west] at (3.4,3.4) {$Z_4$};
\end{tikzpicture} 

\end{figure}
The coefficients of the block tensor $\mathbf{C}$ are
\[
X_{4}=(A_{1}^{\dagger}\otimes A_{3}^{\dagger})\mathbf{X}_{2},\quad Z_{2}=(U_{1}^{\dagger}\otimes U_{3}^{\dagger})U_{2}^{1,3},\quad U_{4}=\dot{V}_{1},
\]
\[
W_{4}=\dot{U}_{2},\quad V_{4}=\dot{V}_{3},\quad\bar{V}_{4}=V_{1}^{2}V_{2}^{\dagger}=(U_{1}^{T}\otimes I)U_{2}^{1},
\]
\[
\bar{W}_{2}=(U_{1}^{\dagger}\otimes U_{3}^{\dagger})X_{2}^{1,3},\quad\bar{U}_{4}=(I\otimes U_{3}^{\dagger})U_{2}^{3}.
\]
The inductive step works because by
\[
\left((A_{1}\otimes A_{3})A_{4}^{(k_{1}k_{3})\times k_{2}},\,\mathbf{U}_{4},\,\mathbf{X}_{4}\right)=\left((\begin{array}{ccc}
A_{1} & U_{1} & X_{1}\end{array})\otimes(\begin{array}{ccc}
A_{3} & U_{3} & X_{3}\end{array})\right)\mathbf{C}
\]
we can reduce the parametrization to the matrix case and reproduce
$\mathbf{U}_{4}$ and $\mathbf{X}_{4}$.

\section{Implicit description of the tangent cone}

The tangent cone for the matrix case can be implicitely defined as
the variety
\[
\left\{ \mathcal{X}\in\mathbb{R}^{n\times m}:\operatorname{rank}\left((I-A_{1}A_{1}^{\dagger})\mathcal{X}(I-A_{2}^{\dagger}A_{2})\right)\leq s_{1}\right\} 
\]
where the rank can be bounded by a set of determinants of minors.
Since we have shown in Corollary \ref{cor:IntersectionOfTangentCones}
that the tangent cone of a tensor variety is the intersection of tangent
cones of matrix varieties, the set of defining equations of the tensor
variety is the union of defining equations of matrix varieties of
the appropriate matricizations.

\section{Acknowledgements}

I thank in particular Reinhold Schneider, Max Pfeffer, André Uschmajew, \newline Sebastian Wolf, Benjamin Huber, Jesko Hüttenhain, Paul Breiding, Kathlén
Kohn and Bernd Sturmfels for productive discussions and useful hints.

\newpage{}

\bibliographystyle{plain}

\end{document}